\documentclass{article}

\usepackage{graphicx}

\usepackage[numbers]{natbib}
\usepackage{natbibspacing}
\renewcommand{\bibfont}{\small}

\usepackage{hyperref}
\hypersetup{colorlinks,
linkcolor=black,
citecolor=black}
\usepackage{bbm}
\usepackage[text={130mm,220mm},centering]{geometry}

\usepackage{amssymb,amsmath}
\usepackage{mathrsfs,stmaryrd}

\newcommand{\cblank}{\ \cdot\ }

\newcommand{\ucontents}[2]{\addcontentsline{toc}{#1}{\numberline{}{#2}}}
\newcommand{\ovln}[1]{\overline{#1}}
\newcommand{\such}{:}
\newcommand{\epsln}{\varepsilon}
\newcommand{\integers}{\mathbb{Z}}
\newcommand{\goesto}{\mapsto}
\newcommand{\reals}{\mathbb{R}}
\newcommand{\demph}[1]{\textbf{\textup{#1}}}
\newcommand{\done}{\hfill\ensuremath{\Box}}
\newenvironment{prooflike}[1]{\begin{trivlist}\item\textbf{#1}\ }
{\end{trivlist}}
\newenvironment{proof}{\begin{prooflike}{Proof}}{\end{prooflike}}
\newcommand{\sub}{\subseteq}

\newcommand{\dee}{\,d}
\newcommand{\cln}{\colon}
\DeclareMathOperator{\Vol}{Vol}
\newcommand{\Vold}[1]{\Vol_{#1}}
\newcommand{\onedec}[1]{#1'}
\newcommand{\cvx}[1]{\onedec{\mathscr{K}_{#1}}}
\newcommand{\ecvx}[1]{\mathscr{K}_{#1}}
\newcommand{\gcvx}{\mathscr{K}}
\newcommand{\gcvxr}[1]{\gcvx[#1]}
\newcommand{\Val}[1]{\onedec{\mathrm{Val}_{#1}}}
\newcommand{\Gr}[2]{\onedec{\mathrm{Gr}_{#1, #2}}}
\newcommand{\Graff}[2]{\onedec{\mathrm{Graff}_{#1, #2}}}
\newcommand{\IVdbl}[2]{\onedec{V_{#1, #2}}}
\newcommand{\IV}[1]{\onedec{V_{#1}}}
\newcommand{\ball}[1]{B_{#1}}
\newcommand{\ucube}[1]{C_{#1}}
\DeclareMathOperator{\image}{image}
\newcommand{\otr}[2]{#1_{#2}}
\newcommand{\bdy}{\partial}
\newcommand{\hoaff}[1]{G_{#1}}
\newcommand{\ho}[1]{H_{#1}}
\newcommand{\card}[1]{|#1|}
\newcommand{\iso}{\cong}
\newcommand{\hlf}{{\textstyle\frac{1}{2}}}
\newcommand{\hi}{\mathbf{H}}
\newcommand{\his}[1]{\mathbf{H}^{#1}}
\newcommand{\ls}[2]{[#1, #2]}
\newcommand{\lss}[3]{[#1, #2, #3]}
\newcommand{\cubint}[2]{\llbracket #1, #2 \rrbracket}
\newcommand{\IVe}[1]{V_{#1}}
\newcommand{\Valg}[1]{\mathrm{Val}(#1)}
\newcommand{\nrm}[2]{\Vert #1 \Vert_{#2}}
\newtheorem{thm}{Theorem}[section]
\newtheorem{propn}[thm]{Proposition}
\newtheorem{lemma}[thm]{Lemma}
\newtheorem{cor}[thm]{Corollary}
\newtheorem{preremark}[thm]{Remark}
\newenvironment{remark}{\begin{preremark}\upshape}{\end{preremark}}
\newtheorem{predefn}[thm]{Definition}
\newenvironment{defn}{\begin{predefn}\upshape}{\end{predefn}}
\newtheorem{preexample}[thm]{Example}
\newenvironment{example}{\begin{preexample}\upshape}{\end{preexample}}
\newtheorem{preexamples}[thm]{Examples}
\newenvironment{examples}{\begin{preexamples}\upshape}{\end{preexamples}}

\author{Tom Leinster%
\thanks{School of Mathematics and Statistics, University of Glasgow, Glasgow
G12 8QW, UK; Tom.Leinster@glasgow.ac.uk.  
Supported by an EPSRC Advanced Research Fellowship}}
\title{\vspace*{-2em}Integral geometry for the 1-norm}
\date{}

\begin{document}

\sloppy
\maketitle

\begin{abstract}  
Classical integral geometry takes place in Euclidean space, but one can
attempt to imitate it in any other metric space.  In particular, one can
attempt this in $\reals^n$ equipped with the metric derived from the $p$-norm.
This has, in effect, been investigated intensively for $1 < p < \infty$, but
not for $p = 1$.  We show that integral geometry for the $1$-norm bears a
striking resemblance to integral geometry for the $2$-norm, but is radically
different from that for all other values of $p$.  We prove a Hadwiger-type
theorem for $\reals^n$ with the $1$-norm, and analogues of the classical
formulas of Steiner, Crofton and Kubota.  We also prove principal and higher
kinematic formulas.  Each of these results is closely analogous to its
Euclidean counterpart, yet the proofs are quite different.
\end{abstract}

\tableofcontents

\section*{Introduction}
\ucontents{section}{Introduction}

Classical integral geometry provides definite answers to natural questions
about convex subsets of Euclidean space.  The Cauchy formula, for instance,
tells us that the surface area of a convex body in $\reals^3$ is proportional
to the expected area of its projection onto a random plane.  The Crofton
formula states that it is also proportional to the measure of the set of
affine lines that meet the body.  The Steiner formula gives the volume
of the set of points within a specified distance of a given convex body.  The
kinematic formula tells us the probability that a randomly-placed convex body
$X$ meets another body $Y$, given that it meets a larger body $Z \supseteq Y$.

It is so successful a theory that one naturally seeks to imitate it elsewhere.
This has been done in several ways.  For example,
Alesker~\cite{AlesTVM,AlesVMI} (foreshadowed by Fu~\cite{FuKFI}) has developed
integral geometry for manifolds, while Bernig and Fu~\cite{BeFu} have
developed Hermitian integral geometry.  Others have extended integral geometry
to finite-dimensional real Banach spaces, and more generally to projective
Finsler spaces: see for instance Schneider and Wieacker~\cite{ScWi} and
Schneider~\cite{SchnIGP,SchnCMP}.  (This includes $\reals^n$ with the
$1$-norm~\cite{SchnCMP}, but it is a different generalization from that
presented here.)  An important role is played there by Holmes-Thompson
valuations, comparable to intrinsic volumes in Euclidean space: see
Schneider~\cite{SchnIVM}, \'Alvarez Paiva and Fernandes~\cite{AlFe}, and
Bernig~\cite{Bern}.

But one simple setting in which integral geometry seems not to have been fully
developed is that of ordinary metric spaces.  A natural notion of convexity is
available there: a subset $X$ of a metric space $A$ is \emph{geodesic} if for
any two points $x, x' \in X$, say distance $D$ apart, there exists an isometry
$\gamma\cln [0, D] \to X$ with $\gamma(0) = x$ and $\gamma(D) = x'$.  Using
this, we can extend to an arbitrary metric space $A$ the fundamental notion of
continuous invariant valuation on convex sets.  And just as for Euclidean
space, the continuous invariant valuations form a vector space, $\Valg{A}$,
indexing all the ways of measuring the size of geodesic subsets of $A$.

Every metric space therefore poses a challenge: classify its continuous
invariant valuations.  A celebrated theorem of Hadwiger answers the challenge
for Euclidean space $\reals^n$, stating that $\Valg{\reals^n}$ is $(n +
1)$-dimensional, with a basis $\IVe{0}, \ldots, \IVe{n}$ in which $\IVe{i}$ is
homogeneous of degree $i$.  The valuations $\IVe{i}$ are the \emph{intrinsic
volumes} (also known, with different normalizations, as the quermassintegrals
or Minkowski functionals).  When $n = 2$, for instance, they are Euler
characteristic, half of perimeter, and area.

More ambitiously, we can attempt to reproduce in an arbitrary metric
space---or one with as little structure as possible---the classical results of
integral geometry in the tradition of Crofton and Blaschke.  For example, we
can seek analogues of the formulas listed in the first paragraph.  To do this,
we will need our metric space to carry an affine structure; and among the most
important such spaces are the Banach spaces $\ell_p^n$, that is, $\reals^n$
equipped with the metric induced by the $p$-norm $\nrm{x}{p} = \bigl( \sum
|x_i|^p \bigr)^{1/p}$ ($p \in [1, \infty)$).

What is known about the integral geometry (in the sense just described) of the
metric spaces $\ell_p^n$?  Let $p \in [1, \infty)$.  The case $p = 2$ is the
classical Euclidean theory.  For $p \neq 1$, the space $\ell_p^n$ is strictly
convex, so the only geodesic subsets are the convex sets.  On the other hand,
for $p \neq 2$, the isometry group of $\ell_p^n$ is very small, generated by
permutations of coordinates, reflections in coordinate hyperplanes, and
translations.  So if $p \neq 1, 2$ then $\ell_p^n$ has the same geodesic
subsets as $\ell_2^n$, but far fewer isometries.  Hence $\Valg{\ell_p^n}$ is
much bigger than $\Valg{\ell_2^n} \cong \reals^{n + 1}$; indeed, it is
infinite-dimensional.  Much is known about the structure of $\Valg{\ell_p^n}$
for $p \neq 1, 2$; this is essentially the theory of translation-invariant
valuations on convex subsets of $\reals^n$ \cite{McMVET, McMCTI, AlesOPM,
AlesDTI}.

But the case $p = 1$ has until now been overlooked, and turns out to contain a
surprise.  As we shall see, the metric space $\ell_1^n$ behaves very much
like $\ell_2^n$, but very much unlike all the other spaces $\ell_p^n$.  For
example, there is a Hadwiger-type theorem stating that $\Valg{\ell_1^n} \cong
\reals^{n + 1}$.  Furthermore, $\Valg{\ell_1^n}$ has a basis $\IV{0}, \ldots,
\IV{n}$ of valuations, the \emph{$\ell_1$-intrinsic volumes}, where $\IV{i}$
is homogeneous of degree $i$.  Hence there is a \emph{canonical} isomorphism
$\Valg{\ell_1^n} \cong \Valg{\ell_2^n}$---despite the fact that $\ell_1^n$ and
$\ell_2^n$ have neither the same geodesic subsets nor the same isometry group.

The resemblance goes deeper still: as we demonstrate, all the standard
Euclidean integral-geometric formulas have close analogues in $\ell_1^n$.
Nevertheless, the proofs are quite different: just as the classical proofs
exploit special features of Euclidean geometry, ours exploit special features
of $\ell_1$ geometry.  

A mystery remains: why are the results for $\ell_1^n$ and $\ell_2^n$
so similar to each other, yet so different from those for $\ell_p^n$ when $p
\in (1, 2) \cup (2, \infty)$?  There is no obvious common generalization of
the cases $p = 1$ and $p = 2$.  Yet a common generalization must surely
exist.

The case $p = \infty$ appears not to have been investigated either.
Since $\ell_\infty^2$ is isometric to $\ell_1^2$, the vector space
$\Valg{\ell_\infty^2}$ is $3$-dimensional, like $\Valg{\ell_1^2}$ and
$\Valg{\ell_2^2}$ but unlike $\Valg{\ell_p^2}$ for $p \in (1, 2) \cup (2,
\infty)$.  It is natural to conjecture that $\Valg{\ell_\infty^n} \cong
\reals^{n + 1}$ for all $n \geq 0$.

This paper is organized as follows.  We begin by establishing the fundamental
facts about geodesic subsets of $\ell_1^n$, here called \emph{$\ell_1$-convex}
sets.  (They include the convex sets, but are much more general.)  Almost
immediately we encounter a stark difference between $\ell_1^n$ and $\ell_2^n$:
the intersection of $\ell_1$-convex sets need not be $\ell_1$-convex.  And
yet, there is a more subtle sense in which the two situations are precisely
analogous (Remark~\ref{rmk:cogeodesic}).  Guided by this analogy, we prove
$\ell_1$ versions of all the elementary laws governing intersections,
projections and Minkowski sums of ordinary convex sets (Sections~\ref{sec:con}
and~\ref{sec:Min}).  We also prove a result that has no clear analogue in
$\ell_2^n$: if the union of two $\ell_1$-convex sets is $\ell_1$-convex, then
so is its intersection.

Having described the algebra of the space of $\ell_1$-convex sets, we turn to
its topology (Sections~\ref{sec:app} and~\ref{sec:val}).  We show that it has
a dense subspace consisting, roughly, of the $\ell_1$-convex unions of cubes.
We then generalize the theorem of McMullen~\cite{McMVET} that a monotone
translation-invariant valuation on convex sets is continuous.  Our
generalization implies both McMullen's theorem and its $\ell_1$ counterpart.

These results provide the tools that enable us to develop the integral
geometry of the metric space $\ell_1^n$.  The $\ell_1$-intrinsic volumes are
defined by a Cauchy-type formula, adapted to the smaller isometry group of
$\ell_1^n$.  We prove analogues of the core theorems of Euclidean integral
geometry: first a Hadwiger-type theorem, then analogues of the Steiner,
Crofton, Kubota and kinematic formulas (Sections
\mbox{\ref{sec:Had}--\ref{sec:kin}}).  While Sections
\ref{sec:con}--\ref{sec:val} depend heavily on specific features of the
geometry of $\ell_1^n$, Sections \ref{sec:Had}--\ref{sec:kin} are formally
close to their Euclidean counterparts.  Even the constants appearing in the
formulas are analogous: one simply replaces the flag coefficients~\cite{KlRo}
in the Euclidean formulas by the corresponding binomial coefficients.  

As this suggests, the integral geometry of $\ell_1^n$ can be regarded as a
cousin of the integral geometry of Euclidean space.  It is more simple
analytically, because of the smaller isometry group of $\ell_1^n$.  But since
there are many more geodesic sets in $\ell_1^n$ than in $\ell_2^n$,
it is also more complex geometrically.

\paragraph*{Conventions} $\reals^n$ denotes real $n$-dimensional space as a
set, topological space or vector space, but with no implied choice of metric
except when $n = 1$.  We allow $n = 0$.  Lebesgue measure on $\reals^n$ is
written as $\Vold{n}$ or $\Vol$.  The metric on a metric space is usually
written as $d$.

\paragraph*{Acknowledgements} I thank Andreas Bernig, Joseph Fu, Daniel Hug,
Mark Meckes, Rolf Schneider and the anonymous referee for helpful
conversations and suggestions.

\section{$\ell_1$-convexity}
\label{sec:con}

Here we define $\ell_1$-convexity and give some useful equivalent conditions.
We also discuss the class of intervals in $\reals^n$, dual in a certain sense
to the class of $\ell_1$-convex sets.  Along the way, we review
some standard material on abstract metric spaces; this can be found in texts
such as Gromov~\cite[Chapter 1]{GroMSR} and Papadopoulos~\cite{Papa}.

\begin{defn}
A \demph{path} in a metric space $X$ is a continuous map $\gamma\cln [c, c']
\to X$, where $c$ and $c'$ are real numbers with $c \leq c'$; it \demph{joins}
$\gamma(c)$ and $\gamma(c')$.  It is \demph{distance-preserving} if
$d(\gamma(t), \gamma(t')) = |t - t'|$ for all $t, t' \in [c, c']$.
\end{defn}

\begin{defn}
A metric space $X$ is \demph{geodesic} if for all $x, x' \in X$, there exists
a distance-preserving path joining $x$ and $x'$.
\end{defn}

For example, a subspace of Euclidean space is geodesic if and only if it is
convex. 

\begin{defn}
A subset of $\reals^n$ is \demph{$\ell_1$-convex} if it is geodesic when
given the subspace metric from $\ell_1^n$.  
\end{defn}

A convex subset of $\reals^n$ is $\ell_1$-convex, but not conversely.  For
example, let $f\cln \reals \to \reals$ be an increasing continuous function:
then the graph $\{(x, f(x)) \such x \in \reals \} \sub \ell_1^2$ is
$\ell_1$-convex.  An $\ell_1$-convex set need not even have positive reach:
consider an L-shaped subset of $\ell_1^2$.

\begin{defn}
Let $\gamma\cln [c, c'] \to X$ be a path in a metric space $X$.  The
\demph{length} of $\gamma$ is the supremum of $\sum_{r = 1}^k d(\gamma(t_{r -
1}), \gamma(t_r))$ over all partitions $c = t_0 \leq t_1 \leq \cdots \leq t_k
= c'$.
\end{defn}

When speaking of functions $[c, c'] \to \reals$, we use the terms
\demph{increasing} and \demph{decreasing} in the non-strict sense.  (For
example, a constant function is both.)  A real-valued function is
\demph{monotone} if it is increasing or decreasing.

\begin{defn}
A path $\gamma\cln [c, c'] \to \reals^n$ is \demph{monotone} if each of its
components $\gamma_i\cln [c, c'] \to \reals$ ($1 \leq i \leq n$) is monotone.
\end{defn}

Let $X$ be a metric space and $x, x' \in X$.  A point $y \in X$ is
\demph{between} $x$ and $x'$ if $d(x, x') = d(x, y) + d(y, x')$, and
\demph{strictly between} if also $x \neq y \neq x'$.  When $X =
\ell_1^n$, a point $y$ is between $x$ and $x'$ if and only
if for all $i \in \{1, \ldots, n\}$, either $x_i \leq y_i \leq x'_i$ or $x'_i
\leq y_i \leq x_i$.

\begin{defn}
A metric space is \demph{Menger convex} if for all distinct points $x, x'$
there exists a point strictly between $x$ and $x'$.
\end{defn}

We will repeatedly use the following characterization theorem for
$\ell_1$-convex sets.

\begin{propn}   \label{propn:lc-eqv}
Let $X \sub \ell_1^n$.  The following are equivalent:
\begin{enumerate}
\item   \label{part:lc-eqv-lc}
$X$ is $\ell_1$-convex
\item   \label{part:lc-eqv-length} 
every pair $x, x'$ of points of $X$ can be joined by a path of length $d(x,
x')$
\item   \label{part:lc-eqv-mono}
every pair of points of $X$ can be joined by a monotone path in $X$.
\end{enumerate}
When $X$ is closed, a further equivalent condition is that $X$ is Menger
convex. 
\end{propn}

\begin{proof}
A metric space is geodesic if and only if each pair of points $x, x'$ can be 
joined by a path of length $d(x, x')$~\cite[Proposition 2.2.7]{Papa}.  This
proves the equivalence of~(\ref{part:lc-eqv-lc})
and~(\ref{part:lc-eqv-length}).

For (\ref{part:lc-eqv-length})$\Leftrightarrow$(\ref{part:lc-eqv-mono}), a path
$\gamma\cln [c, c'] \to \ell_1^n$ has length $d(\gamma(c), \gamma(c'))$ if and
only if
\begin{equation}
\label{eq:straight-sum}
|\gamma_i(t) - \gamma_i(t')| 
+
|\gamma_i(t') - \gamma_i(t'')|
=
|\gamma_i(t) - \gamma_i(t'')|
\end{equation}
whenever $1 \leq i \leq n$ and $c \leq t \leq t' \leq t'' \leq c'$.
Equation~\eqref{eq:straight-sum} holds if and only if $\gamma_i(t')$ is
between $\gamma_i(t)$ and $\gamma_i(t'')$.  It follows that $\gamma$ has
length $d(\gamma(c), \gamma(c'))$ if and only if it is monotone.

The final statement follows from the fact that for a metric space
in which every closed bounded set is compact, Menger convexity is equivalent
to being geodesic \cite[Theorem 2.6.2]{Papa}.  \done
\end{proof}

The intersection of $\ell_1$-convex sets need not be $\ell_1$-convex, and can
in fact be highly irregular.  For example, every closed subset of the line
occurs, up to isometry, as the intersection of a pair of $\ell_1$-convex
subsets of the plane.  Indeed, if $\emptyset \neq K \sub \reals$ is closed
then the sets
\[
\bigl\{ 
{\textstyle\frac{1}{2}}\bigl(x - d(x, K), x + d(x, K)\bigr) 
\such
x \in \reals 
\bigr\}
\sub \ell_1^2,
\qquad
\bigl\{ 
{\textstyle\frac{1}{2}}(x, x) 
\such
x \in \reals 
\bigr\}
\sub \ell_1^2
\]
are $\ell_1$-convex and have intersection isometric to $K$.  We do, however,
have the following.

\begin{defn}
An \demph{interval} in $\reals^n$ is a subset of the form $\prod_{i = 1}^n
I_i$ for some (possibly empty, possibly unbounded) intervals $I_1, \ldots, I_n
\sub \reals$.
\end{defn}

\begin{cor}     \label{cor:intersection}
The intersection of an $\ell_1$-convex set and an interval in $\reals^n$ is
$\ell_1$-convex.
\end{cor}

\begin{proof}
An interval $I$ has the property that whenever $x, x' \in I$, every
monotone path from $x$ to $x'$ in $\reals^n$ lies in $I$.  The result follows
from Proposition~\ref{propn:lc-eqv}(\ref{part:lc-eqv-mono}).
\done
\end{proof}

\begin{remark}  \label{rmk:cogeodesic}
Corollary~\ref{cor:intersection} might seem weak when compared with the result
in Euclidean space that the intersection of \emph{any} pair of convex sets is
convex.  But in fact, the $\ell_1$ and Euclidean results are strictly
analogous---as long as one uses the correct analogy.  Let $A$ be a metric
space and, for $a, a' \in A$, write $\Gamma(a, a')$ for the set of
distance-preserving paths joining $a$ and $a'$.  For a subspace $X \sub A$ to
be geodesic means that whenever $x, x' \in X$,
\[
\exists \gamma \in \Gamma(x, x')\cln
\image(\gamma) \sub X.
\]
There is a dual condition: $X \sub A$ is \demph{cogeodesic} if whenever $x, x'
\in X$,
\[
\forall \gamma \in \Gamma(x, x'),\ 
\image(\gamma) \sub X.
\]
A subset of $\ell_2^n$ is cogeodesic if and only if it is convex, if and only
if it is geodesic, but a subset of $\ell_1^n$ is cogeodesic if and only if it
is an interval.  It is a logical triviality that in any metric space, the
intersection of a geodesic subset and a cogeodesic subset is geodesic.
Applied to $\ell_2^n$, this says that the intersection of two convex sets is
convex.  Applied to $\ell_1^n$, this is Corollary~\ref{cor:intersection}.
\end{remark}

A subset $X \sub \reals^n$ is \demph{orthogonally convex}~\cite{FiWo} if $X
\cap L$ is convex whenever $L$ is a straight line parallel to one of the
coordinate axes.  Corollary~\ref{cor:intersection} has the following special
case:

\begin{cor}     \label{cor:orthoconvex}
An $\ell_1$-convex set is orthogonally convex.
\done
\end{cor}

On the other hand, an orthogonally convex set need not be $\ell_1$-convex,
even if it is connected.  For example, choose a vector $v \in \reals^3$ none
of whose coordinates is $0$, and consider the set of unit-length vectors in
$\ell_2^3$ orthogonal to $v$.  

A \demph{coordinate subspace} of $\reals^n$ is a linear subspace $P$ spanned
by some subset of the standard basis.  We write $\pi_P$ for the orthogonal
projection of $\reals^n$ onto $P$, and $P^\perp$ for the orthogonal complement
of $P$ (with respect to the standard inner product).  By
Proposition~\ref{propn:lc-eqv}(\ref{part:lc-eqv-mono}), we have:

\begin{cor}     \label{cor:pjn-pres}
Let $P$ be a coordinate subspace of $\reals^n$.  Then the image under $\pi_P$
of an $\ell_1$-convex set is $\ell_1$-convex.
\done
\end{cor}

There is a further positive result on intersections of $\ell_1$-convex sets.

\begin{lemma}   \label{lemma:cup-cap}
Let $X$ and $Y$ be closed subsets of $\reals^n$.  If $X$, $Y$ and $X \cup Y$
are $\ell_1$-convex, then so is $X \cap Y$.
\end{lemma}

\begin{proof}
We use the following property of $\ell_1^n$: if a point $a$ is between points
$x$ and $y$, and if $x$ and $y$ are both strictly between points $z$ and $z'$,
then $a$ is strictly between $z$ and $z'$.

We prove that $X \cap Y$ is Menger convex.  Let $z, z' \in X \cap Y$ with $z
\neq z'$.  Since $X$ and $Y$ are Menger convex, we may choose points $x \in X$
and $y \in Y$ strictly between $z$ and $z'$.  Since $X \cup Y$ is
$\ell_1$-convex, we may choose a distance-preserving path $\gamma\cln [c, c']
\to X \cup Y$ joining $x$ and $y$.  Since $X$ and $Y$ are closed, we may
choose $t \in [c, c']$ with $\gamma(t) \in X \cap Y$.  Then $\gamma(t)$ is
between $x$ and $y$, hence strictly between $z$ and $z'$, as required.  \done
\end{proof}

The interior and closure of a convex set are convex.  The interior of an
$\ell_1$-convex set need not be $\ell_1$-convex: consider $[-1, 0]^2 \cup [0,
1]^2 \sub \reals^2$.  On the other hand, we have the following.

\begin{lemma}   \label{lemma:closure}
The closure of an $\ell_1$-convex set is $\ell_1$-convex.
\end{lemma}

\begin{proof}
We prove that the closure $\ovln{X}$ of an $\ell_1$-convex set $X$ is Menger
convex.  Let $x, y \in \ovln{X}$.  Choose sequences
$(x_r)$ and $(y_r)$ in $X$ converging to $x$ and $y$.  Choose for each $r$ a
point $z_r \in X$ with $d(x_r, z_r) = d(z_r, y_r) = d(x_r, y_r)/2$.  The
sequence $(z_r)$ is bounded, so has a subsequence convergent to some point $z
\in \ovln{X}$.  Then $d(x, z) = d(z, y) = d(x, y)/2$.
\done
\end{proof}

\section{Minkowski sums}
\label{sec:Min}

In the Euclidean context~\cite{SchnCBB}, there are basic laws
governing the algebra of intersections and Minkowski sums: (i)~if $X$ and $I$
are convex then so is $X \cap I$; (ii)~if $X$ and $I$ are convex then so is $X
+ I$; and (iii)~if $X, Y$ are closed with $X \cup Y$ convex, and $I$ is
convex, then $(X \cap Y) + I = (X + I) \cap (Y + I)$.

In the $\ell_1$ context, we already have an analogue of~(i)
(Corollary~\ref{cor:intersection}).  Here we prove analogues of~(ii) and~(iii).
As in Remark~\ref{rmk:cogeodesic}, the analogy entails replacing some
occurrences of the term `convex set' by `$\ell_1$-convex set', and others by
`interval'.

First we note that the class of $\ell_1$-convex sets is not closed under
Minkowski sums.

\begin{example}
Given $x, y \in \reals^n$, write $\ls{x}{y}$ for the closed straight line
segment between $x$ and $y$.  Given also $z \in \reals^n$, write $\lss{x}{y}{z}
= \ls{x}{y} \cup \ls{y}{z}$.  Define $X, Y \sub \reals^3$ by 
\[
X = \lss{(0, 0, 0)}{(2, 0, 0)}{(2, 2, -1)},
\quad
Y = \lss{(0, 0, 0)}{(0, -1, 2)}{(-1, -1, 2)}.
\]
Then $X$ and $Y$ are $\ell_1$-convex, but $X + Y$ is not.  Indeed, $(0, 0, 0)$
and $(1, 1, 1)$ are points of $X + Y$ distance $3$ apart in $\ell_1^3$, but
there is no point of $X + Y$ distance $3/2$ from each of them.
\end{example}

To prove our analogue of~(ii), we use the following sufficient condition for
$\ell_1$-convexity of a Minkowski sum.

\begin{lemma}   \label{lemma:plus-plus}
Let $X \sub \reals^n$ be a closed set and $I \sub \reals^n$ a compact
interval.  Suppose that for every $x, x' \in X$ satisfying $(x + I) \cap (x' +
I) = \emptyset$, there exists a point of $X$ strictly between $x$ and $x'$ in
the $\ell_1$ metric.  Then $X + I$ is $\ell_1$-convex.
\end{lemma}

\begin{proof}
The topological hypotheses imply that $X + I$ is closed, so by
Proposition~\ref{propn:lc-eqv}, it is enough to prove that $X + I$ is Menger
convex.  Let $y$ and $y'$ be distinct points of $X + I$.  Write
$\cubint{y}{y'}$ for the interval consisting of the points between $y$ and
$y'$.  Since $X$ is closed and $I$ is compact, we may choose $x, x' \in X$
such that $y \in x + I$, $y' \in x' + I$, and $d(x, x')$ is minimal for all
such pairs $(x, x')$.

The proof is in two cases.  First suppose that $(x + I) \cap (x' + I) =
\emptyset$.  By hypothesis, we may choose a point $z \in X$ strictly between
$x$ and $x'$.  By 
minimality, $y \not\in z + I$ and $y' \not\in z + I$.  Also, $y \in x + I$, $y'
\in x' + I$, and $z$ is between $x$ and $x'$, from which it follows that
$\cubint{y}{y'} \cap (z + I) \neq \emptyset$.  Any point in this intersection
is strictly between $y$ and $y'$.  

Now suppose that $(x + I) \cap (x' + I) \neq \emptyset$.  If $y$ or $y'$ is in
$(x + I) \cap (x' + I)$ then $(y + y')/2$ is a point of $X + I$ strictly
between $y$ and $y'$.  If not, it is enough to prove that
\[
\cubint{y}{y'} \cap (x + I) \cap (x' + I) \neq \emptyset.
\]
This follows from the fact that if $J_1, J_2, J_3$ are intervals in $\reals^n$
whose pairwise intersections are all nonempty, then $J_1 \cap J_2 \cap J_3$ is
also nonempty.  \done
\end{proof}

\begin{propn}   \label{propn:Min}
The Minkowski sum of a closed $\ell_1$-convex set and an interval
is $\ell_1$-convex. 
\end{propn}

\begin{proof}
Let $X \sub \reals^n$ be a closed $\ell_1$-convex set, and let $I \sub
\reals^n$ be an interval.  We may write $I$ as a union of compact subintervals
$I^1 \sub I^2 \sub \cdots$.  By Lemma~\ref{lemma:plus-plus}, $X + I^r$ is
$\ell_1$-convex for each $r \geq 1$.  But the class of $\ell_1$-convex sets is
closed under nested unions, so $\bigcup_r (X + I^r) = X + I$ is
$\ell_1$-convex.  \done
\end{proof}

Here is our analogue of law~(iii).  The proof is similar to the proof
of the Euclidean case (Lemma~3.1.1 of~\cite{SchnCBB}).

\begin{propn}   \label{propn:Min-cap}
Let $X, Y, I \sub \reals^n$.  Then
\[
(X \cup Y) + I = (X + I) \cup (Y + I).
\]
If $X$ and $Y$ are closed with $X \cup Y$ $\ell_1$-convex, and $I$ is an
interval, then also
\[
(X \cap Y) + I = (X + I) \cap (Y + I).
\]
\end{propn}

\begin{proof}
The first equation is trivial.  In the second, the left-hand side is certainly
a subset of the right-hand side.  For the converse, let $z \in (X + I) \cap (Y
+ I)$, writing
\[
z = x + a = y + b
\]
($x \in X$, $y \in Y$, $a, b \in I$).  Choose a monotone path $\gamma\cln [0,
1] \to X \cup Y$ joining $x$ and $y$.  Define a path $\alpha\cln [0, 1] \to
\reals^n$ by $\alpha(t) = z - \gamma(t)$.  Since $I$ is an interval and
$\alpha$ is a monotone path whose endpoints are in $I$, the whole image of
$\alpha$ lies in $I$.  Since $X$ and $Y$ are closed, there exists $t \in [0,
1]$ such that $\gamma(t) \in X \cap Y$.  Then $z = \gamma(t) + \alpha(t) \in
(X \cap Y) + I$, as required.  \done
\end{proof}

Example~\ref{eg:ball-Steiner} shows that the second part of
Proposition~\ref{propn:Min-cap} can fail when $I$ is merely $\ell_1$-convex.

\begin{cor}   \label{cor:pjn-cap}
Let $X, Y \sub \reals^n$ and let $P$ be a coordinate subspace of $\reals^n$.
Then
\[
\pi_P(X \cup Y) = \pi_P X \cup \pi_P Y.
\]
If $X$ and $Y$ are closed and $X \cup Y$ is $\ell_1$-convex then also
\[
\pi_P(X \cap Y) = \pi_P X \cap \pi_P Y.
\]
\end{cor}

\begin{proof}
This follows from Proposition~\ref{propn:Min-cap}, since $\pi_P Z = (Z +
P^\perp) \cap P$ for all $Z \sub \reals^n$. 
\ \mbox{}\ \done
\end{proof}

\section{Approximation of $\ell_1$-convex sets}
\label{sec:app}

Essential to our proof of the $\ell_1$ Hadwiger theorem is the result, proved
in this section, that a compact $\ell_1$-convex set can be approximated
arbitrarily well by a finite union of cubes \emph{that is itself
$\ell_1$-convex}.

Write $\ucube{n} = [-\hlf, \hlf]^n$ for the unit $n$-cube, and $\hi = \{ m +
\hlf \such m \in \integers\}$ for the set of half-integers.  Given $X \sub
\reals^n$ and $\lambda \geq 0$, write $\lambda X = \{\lambda x \such x \in
X\}$.  When $\lambda > 0$, write
\[
\otr{X}{\lambda}
=
\bigcup
\bigl\{ 
\lambda (h + \ucube{n})
\such 
h \in \his{n} \text{ with } 
\lambda(h + \ucube{n}) \cap X \neq \emptyset 
\bigr\}
\supseteq
X.
\]

\begin{propn}   \label{propn:lc-outer}
Let $X \sub \reals^n$ be an $\ell_1$-convex set and $\lambda > 0$.  Then
$\otr{X}{\lambda}$ is $\ell_1$-convex.
\end{propn}

\begin{proof}
Assume without loss of generality that $\lambda = 1$.  Put 
\[
L 
=
\{ h \in \his{n}
\such
(h + \ucube{n}) \cap X \neq \emptyset \},
\]
so that $X_1 = L + \ucube{n}$.  We show that the closed set $L$ and the
interval $\ucube{n}$ satisfy the conditions of Lemma~\ref{lemma:plus-plus}.
The result will follow.

Let $h, h' \in L$ with $(h + \ucube{n}) \cap (h' + \ucube{n}) = \emptyset$,
assuming without loss of generality that $h_1 + \hlf < h'_1 - \hlf$.  Choose
$x \in (h + \ucube{n}) \cap X$ and $x' \in (h' + \ucube{n}) \cap X$.  Then
\[
x_1 \leq h_1 + \hlf < h'_1 - \hlf \leq x'_1.
\]
By $\ell_1$-convexity, there exists a monotone path from $x$ to $x'$ in $X$;
this contains a point $z$ with $h_1 + \hlf < z_1 < h'_1 - \hlf$.
Since $z$ is between $x$ and $x'$, we have $z \in k + \ucube{n}$ for some $k
\in \his{n}$ between $h$ and $h'$.  Then $k \in L$.  The constraints on $z_1$
force $h_1 < k_1 < h'_1$, so $h \neq k \neq h'$. 
\done
\end{proof}

For $\lambda > 0$, let us say that a subset of $\reals^n$ is
\demph{$\lambda$-pixellated} if it is a finite union of cubes of the form
$\lambda(h + \ucube{n})$ ($h \in \his{n}$).  A set is \demph{pixellated} if it
is $\lambda$-pixellated for some $\lambda > 0$.  The role of pixellated sets
in the $\ell_1$ theory is similar to that of polyhedra in the Euclidean
theory.  Proposition~\ref{propn:lc-outer} implies:

\begin{thm}   \label{thm:dense}
The set of pixellated $\ell_1$-convex subsets of $\reals^n$ is dense in the
space of compact $\ell_1$-convex subsets of $\reals^n$, with respect to the
Hausdorff metric.  \done
\end{thm}

In a later proof, we will use a hyperplane to divide a pixellated
$\ell_1$-convex set into two smaller sets.  We will need the following lemma.

\begin{lemma}   \label{lemma:divide}
Let $\lambda > 0$ and let $X \sub \reals^n$ be a $\lambda$-pixellated
$\ell_1$-convex set.  Write
\begin{align*}
X^+     &=
\bigcup 
\bigl\{ 
\lambda(h + \ucube{n})
\such
h \in \his{n} 
\text{ with } \lambda(h + \ucube{n}) \sub X
\text{ and } h_1 > 0 
\bigr\},        \\
X^-     &=
\bigcup 
\bigl\{ 
\lambda(h + \ucube{n})
\such
h \in \his{n} 
\text{ with } \lambda(h + \ucube{n}) \sub X
\text{ and } h_1 < 0 
\bigr\}.
\end{align*}
Then $X^+$, $X^-$ and $X^+ \cap X^-$ are all $\ell_1$-convex.
\end{lemma}

\begin{proof}
$X^+$ is the closure of $X \cap ((0, \infty) \times \reals^{n - 1})$, so is
$\ell_1$-convex by Corollary~\ref{cor:intersection} and
Lemma~\ref{lemma:closure}.  Similarly, $X^-$ is $\ell_1$-convex.  Now $X^+
\cup X^- = X$, and $X$ is $\ell_1$-convex, so $X^+ \cap X^-$ is
$\ell_1$-convex by Lemma~\ref{lemma:cup-cap}.  
\done
\end{proof}

The following result will not be needed later, but is of independent interest.
It generalizes the classical fact that compact convex sets are Jordan
measurable.

\pagebreak

\begin{propn}
Every compact $\ell_1$-convex subset of $\reals^n$ is Jordan measurable.
\end{propn}

\begin{proof}
Let $X \sub \reals^n$ be a compact $\ell_1$-convex set.  For
$\lambda > 0$, write
\[
D(\lambda)
=
\bigcup 
\bigl\{ \lambda(h + \ucube{n})
\such
h \in \his{n} 
\text{ with } \lambda(h + \ucube{n}) \cap X \neq \emptyset
\text{ and } \lambda(h + \ucube{n}) \not\sub X \bigr\}.
\]
Write $\bdy X$ for the topological boundary of $X$.  Then $\bdy X \sub
D(\lambda)$ for all $\lambda > 0$, and we have to show that $\Vol(\bdy X) =
0$.

For $\lambda > 0$ and $h \in \his{n}$, the set $D(\lambda)$ cannot contain all
$3^n$ of the cubes $\lambda(h + \sigma + \ucube{n})$ with $\sigma \in \{-1,
0, 1\}^n$.  Indeed, if $D(\lambda)$ contains all $2^n$ corner cubes then $X$
has nonempty intersection with all the corner cubes, and then it follows
from $\ell_1$-convexity that $X$ contains the whole central cube $\lambda(h +
\ucube{n})$; hence $D(\lambda)$ does not contain the central cube.

We therefore have
\[
\Vol(D(\lambda)) 
\leq
((3^n - 1)/3^n)\Vol(D(3\lambda))
\]
for all $\lambda > 0$.  Thus, $\inf_{\lambda > 0} \Vol(D(\lambda)) = 0$,
giving $\Vol(\bdy X) = 0$.
\done
\end{proof}

\section{Valuations}
\label{sec:val}

Here we generalize the theorem of McMullen~\cite[Theorem 8]{McMVET} that a
monotone translation-invariant valuation on convex sets is continuous.  This
will make numerous continuity checks very easy.

For the rest of this section, let $\gcvx$ be a set of compact, orthogonally
convex subsets of $\reals^n$.  We suppose that $\gcvx$ contains all
singletons $\{x\}$, and that $X + I \in \gcvx$ whenever $X \in \gcvx$ and $I$
is a compact interval.  Then $\gcvx$ contains all compact intervals and is
closed under translations.  For example, $\gcvx$ might be the set $\ecvx{n}$
of compact convex subsets of $\reals^n$, or the set $\cvx{n}$ of compact
$\ell_1$-convex subsets of $\reals^n$ (by Corollary~\ref{cor:orthoconvex} and
Proposition~\ref{propn:Min}).

A \demph{valuation} on $\gcvx$ is a function $\phi\cln \gcvx \to \reals$
such that $\phi(\emptyset) = 0$ and
\[
\phi(X \cup Y) 
=
\phi(X) + \phi(Y) - \phi(X \cap Y)
\]
whenever $X, Y, X \cup Y, X \cap Y \in \gcvx$.  (The hypothesis that $X \cap Y
\in \gcvx$ is redundant for $\ecvx{n}$, and also for $\cvx{n}$, by
Lemma~\ref{lemma:cup-cap}.)

We give $\gcvx$ the Hausdorff metric induced by the $\ell_1$ metric on
$\reals^n$.  This can be defined as follows.  Writing $\ball{n}$ for
the closed unit ball of $\ell_1^n$, the \demph{Hausdorff distance} between
compact sets $X, Y \sub \reals^n$ is
\[
d(X, Y) 
=
\inf \{ \delta > 0 \such 
X \sub Y + \delta\ball{n} \text{ and } Y \sub X + \delta\ball{n} \}.
\]
By compactness, the infimum is attained.

A valuation $\phi$ is \demph{continuous} if it is continuous with
respect to the Hausdorff metric.  The notion of continuity is unaffected by
the choice of the 1-norm over other norms on $\reals^n$, since all such are
equivalent.  A valuation $\phi$ is \demph{increasing} if $X \sub Y$ implies
$\phi(X) \leq \phi(Y)$ for $X, Y \in \gcvx$, and \demph{monotone} if $\phi$ or
$-\phi$ is increasing.  It is \demph{translation-invariant} if $\phi(X + a) =
\phi(X)$ for all $X \in \gcvx$ and $a \in \reals^n$.

For $R \geq 0$, write $\gcvxr{R} = \{ X \in \gcvx \such X \sub R\ucube{n} \}$.

\begin{lemma}   \label{lemma:shrinking-ball}
Let $\phi$ be a monotone translation-invariant valuation on $\gcvx$.
Let $R \geq 0$.  Then
\[
\lim_{\delta \to 0} \phi(X + \delta \ucube{n}) = \phi(X)
\]
uniformly in $X \in \gcvxr{R}$.
\end{lemma}

\begin{proof}
Given $1 \leq i \leq n$, write $\nu_i\cln \reals \to \reals^n$ for the
embedding of $\reals$ as the $i$th coordinate axis of $\reals^n$.  Given also
$X \in \gcvx$, define $f_{X, i}\cln [0, \infty) \to \reals$ by
\[
f_{X, i}(\delta)
=
\phi\bigl(X + \nu_i[-\delta/2, \delta/2]\bigr) - \phi(X).
\]
For all $\delta, \delta' \geq 0$, by orthogonal convexity of $X$ and the
valuation property, we have
\[
\phi\bigl(X + \nu_i[-\delta, \delta']\bigr)
=
\phi\bigl(X + \nu_i[-\delta, 0]\bigr) 
+ \phi\bigl(X + \nu_i[0, \delta']\bigr) 
- \phi(X).
\]
So by translation-invariance, $f_{X, i}(\delta + \delta') = f_{X, i}(\delta) +
f_{X, i}(\delta')$.  But $f_{X, i}$ is monotone, so $f_{X, i} (\delta) = f_{X,
i}(1) \cdot \delta$ for all $\delta \geq 0$.  Assume without loss of
generality that $\phi$ is increasing.  Then whenever $S \geq 0$ and $X \in
\gcvxr{S}$, we have
\[
0 
\leq
f_{X, i}(1)
\leq 
\phi\bigl(X + \nu_i[-\hlf, \hlf]\bigr)
\leq
\phi((S + 1)\ucube{n}),
\]
so
\[
0 
\leq 
f_{X, i}(\delta) 
\leq 
\phi((S + 1)\ucube{n}) \cdot \delta 
\]
for all $\delta \geq 0$.   

From this estimate and the fact that $\ucube{n} = \sum_{i = 1}^n \nu_i[-\hlf,
\hlf]$, we deduce that
\[
\phi(X)
\leq
\phi(X + \delta \ucube{n})
\leq
\phi(X) + n \phi((R + 2)\ucube{n}) \cdot \delta
\]
for all $X \in \gcvxr{R}$ and $\delta \in [0, 1]$.  The result follows.
\done
\end{proof}

\begin{thm}     \label{thm:mono-cts-gen}
A monotone translation-invariant valuation on $\gcvx$ is continuous.
\end{thm}

\begin{proof}
Consider, without loss of generality, an increasing translation-invariant
valuation $\phi$.  Let $X \in \gcvx$ and $\epsln > 0$.
Choose $R \geq 2$ with $X \sub (R - 2)\ucube{n}$.  By
Lemma~\ref{lemma:shrinking-ball}, we may choose $\eta > 0$ such that $\phi(Y +
\eta\ucube{n}) \leq \phi(Y) + \epsln$ for all $Y \in \gcvxr{R}$.  Put
$\delta = \min\{1, \eta/2\}$.

I claim that $|\phi(Y) - \phi(X)| \leq \epsln$ whenever $Y \in \gcvx$ with
$d(X, Y) \leq \delta$.  Indeed, take such a $Y$.  Then $Y \sub X +
\delta\ball{n}$; but also $\ball{n} \sub 2\ucube{n}$, so $Y \sub X +
2\delta\ucube{n}$.  This implies that $Y \sub R\ucube{n}$, so $Y \in
\gcvxr{R}$, and clearly $X \in \gcvxr{R}$ too.  It also implies that
$X \sub Y + \eta\ucube{n}$, giving
\[
\phi(X) \leq \phi(Y + \eta\ucube{n}) \leq \phi(Y) + \epsln.
\]
Similarly, $Y \sub X + \eta\ucube{n}$, so $\phi(Y) \leq \phi(X) + \epsln$.
This proves the claim.  
\done
\end{proof}

\begin{cor}[McMullen~\cite{McMVET}] 
A monotone translation-invariant valuation on $\ecvx{n}$ is continuous.
\done
\end{cor}

\begin{cor}     \label{cor:mono-cts}
A monotone translation-invariant valuation on $\cvx{n}$ is continuous.
\done
\end{cor}

Lebesgue measure, as a real-valued function on compact subsets of $\reals^n$,
is not continuous with respect to the Hausdorff metric.  It is, however,
continuous when restricted to convex sets \cite[Theorem~12.7]{Val}.
Corollary~\ref{cor:mono-cts} generalizes this classical result to the larger
class of $\ell_1$-convex sets:

\begin{cor}     \label{cor:Vol-cts}
The volume function $\Vol\cln \cvx{n} \to \reals$ is continuous.
\done
\end{cor}

\section{An analogue of Hadwiger's Theorem}
\label{sec:Had}

We are now in a position to prove $\ell_1$ analogues of the classical theorems
of integral geometry.  We begin with Hadwiger's theorem, adopting a strategy
similar in outline to that of Klain~\cite{Klai}.

Denote by $\hoaff{n}$ the isometry group of $\ell_1^n$.  It is generated by
translations, coordinate permutations, and reflections in coordinate
hyperplanes.  A valuation $\phi$ on $\cvx{n}$ is \demph{invariant} if
$\phi(gX) = \phi(X)$ whenever $X \in \cvx{n}$ and $g \in \hoaff{n}$.  The
continuous invariant valuations on $\cvx{n}$ form a vector space $\Val{n}$
over $\reals$.

Given $0 \leq i \leq n$, write $\Gr{n}{i}$ for the set of $i$-dimensional
coordinate subspaces of $\reals^n$; it has $\binom{n}{i}$ elements.  Define
$\IVdbl{n}{i}\cln \cvx{n} \to \reals$, the $i$th \demph{$\ell_1$-intrinsic
volume} on $\reals^n$, by
\[
\IVdbl{n}{i}(X)
=
\sum_{P \in \Gr{n}{i}}
\Vold{i}(\pi_P X)
\]
($X \in \cvx{n}$).  In the $P$-summand, $\Vold{i}$ denotes Lebesgue measure on
$P \iso \reals^i$.

\begin{examples}        \label{egs:ivs}
\begin{enumerate}
\item The $0$th $\ell_1$-intrinsic volume $\IVdbl{n}{0}$ is the Euler
characteristic $\chi$, given by $\chi(\emptyset) = 0$ and $\chi(X) = 1$
whenever $X \in \cvx{n}$ is nonempty.

\item   \label{eg:ivs-cube}
The unit cube $\ucube{n}$ has $\ell_1$-intrinsic volumes
$\IVdbl{n}{i}(\ucube{n}) = \binom{n}{i}$, which are the same as its
Euclidean intrinsic volumes.

\item Write $\ball{n}$ for the unit ball in $\ell_1^n$.  Then $\Vol(\ball{n})
= 2^n/n!$, giving $\IVdbl{n}{i}(\ball{n}) = \frac{2^i}{i!}\binom{n}{i}$.
Taking $n = 2$ and $i = 1$ shows that the $\ell_1$- and Euclidean intrinsic
volumes of a convex set are not always equal (and nor are they equal
up to a constant factor, by the previous example).
\end{enumerate}
\end{examples}

A valuation $\phi$ on $\cvx{n}$ is \demph{homogeneous of degree $i$} if
$\phi(\lambda X) = \lambda^i\phi(X)$ for all $\lambda \geq 0$ and $X \in
\cvx{n}$.  In the following lemma, we write $\nu\cln \reals^n \to \reals^{n +
1}$ for the embedding that inserts $0$ in the last coordinate.

\begin{lemma}   \label{lemma:iv-iv}
Let $0 \leq i \leq n$.  Then:
\begin{enumerate}
\item   \label{item:iv-civ}
$\IVdbl{n}{i}$ is a continuous invariant valuation on $\cvx{n}$

\item   \label{item:iv-hgs}
$\IVdbl{n}{i}$ is homogeneous of degree $i$

\item   \label{item:iv-i}
$\IVdbl{n}{i}(X) = \IVdbl{n + 1}{i}(\nu X)$ for all $X \in \cvx{n}$

\item   \label{item:iv-v}
$\IVdbl{n}{n} = \Vold{n}$

\item   
$\IVdbl{n}{i}$ is increasing.
\end{enumerate}
\end{lemma}

\begin{proof}
The only nontrivial part is~\eqref{item:iv-civ}.  First fix $P \in \Gr{n}{i}$.
By Corollary~\ref{cor:pjn-cap}, the function $X \goesto \Vold{i}(\pi_P X)$ on
$\cvx{n}$ is a valuation.  It is also monotone and translation-invariant, and
therefore continuous by Corollary~\ref{cor:mono-cts}.  This holds for all $P$,
and~\eqref{item:iv-civ} follows.
\done
\end{proof}

Parts~\eqref{item:iv-civ} and~\eqref{item:iv-i} allow us to write
$\IVdbl{n}{i}$ as just $\IV{i}$: for if $X \in \cvx{n}$ and $\reals^n$ is
embedded as a coordinate subspace of some larger space $\reals^N$, then the
$i$th $\ell_1$-intrinsic volume of $X$ is the same whether $X$ is regarded as
a subset of $\reals^n$ or of $\reals^N$.  This justifies the word `intrinsic'.

The \demph{dimension} of a nonempty $\ell_1$-convex set $X \sub \reals^n$ is
the smallest $i \in \{0, \ldots, n\}$ such that $X \sub P + q $ for some $P
\in \Gr{n}{i}$ and $q \in \reals^n$.  A valuation $\psi$ on $\cvx{n}$ is
\demph{simple} if $\psi(X) = 0$ whenever $X \in \cvx{n}$ is of dimension less
than $n$.

\begin{propn}   \label{propn:volume}
Let $\psi$ be a simple continuous translation-invariant valuation on
$\cvx{n}$.  Then $\psi = c\Vold{n}$ for some $c \in \reals$.
\end{propn}

\begin{proof}
Put $\zeta = \psi - \psi(\ucube{n})\Vold{n}$.  By Corollary~\ref{cor:Vol-cts},
$\zeta$ is a simple continuous translation-invariant valuation with
$\zeta(\ucube{n}) = 0$.  We will prove that $\zeta$ is identically zero.

By dividing cubes into smaller cubes, $\zeta(\lambda \ucube{n}) = 0$ for all
rational $\lambda > 0$.  By continuity, $\zeta(\lambda \ucube{n}) = 0$ for all
real $\lambda > 0$.

Now I claim that $\zeta(X) = 0$ whenever $X$ is a pixellated $\ell_1$-convex
set.  Suppose that $X$ is $\lambda$-pixellated, where $\lambda > 0$; thus, $X$
is a union of some finite number $m$ of cubes of the form $\lambda(h +
\ucube{n})$ ($h \in \his{n}$).  If $m = 0$ or $m = 1$ then
$\zeta(X) = 0$ immediately.

Suppose inductively that $m \geq 2$.  Without loss of generality, the $m$
cubes are not all on the same side of the hyperplane $\{ y \in \reals^n \such
y_1 = 0\}$.  Define sets $X^+$ and $X^-$ as in Lemma~\ref{lemma:divide}; thus,
$X = X^+ \cup X^-$, and each of the sets $X^+$, $X^-$, $X^+ \cap X^-$ is
$\ell_1$-convex.  Then $\zeta(X^+) = \zeta(X^-) = 0$ by inductive hypothesis,
and $\zeta(X^+ \cap X^-) = 0$ since $\zeta$ is simple.  Hence $\zeta(X) = 0$
by the valuation property, completing the induction and proving the claim.

The result now follows from Theorem~\ref{thm:dense}.
\done
\end{proof}

\begin{thm}     \label{thm:Had}
The $\ell_1$-intrinsic volumes $\IV{0}, \ldots, \IV{n}$ form a basis for the
vector space $\Val{n}$ of continuous invariant valuations on $\ell_1$-convex
subsets of $\reals^n$.  In particular, $\dim(\Val{n}) = n + 1$.
\end{thm}

\begin{proof}
By Lemma~\ref{lemma:iv-iv}(\ref{item:iv-i}, \ref{item:iv-v}), we have
$\IV{i}(\ucube{i}) = 1$ whenever $0 \leq i \leq n$, and $\IV{j}(\ucube{i}) =
0$ whenever $0 \leq i < j \leq n$.  It follows that $\IV{0}, \ldots, \IV{n}$
are linearly independent.

We prove by induction on $n$ that $\IV{0}, \ldots, \IV{n}$ span $\Val{n}$.
This is trivial when $n = 0$.  Suppose that $n \geq 1$, and let $\phi \in
\Val{n}$.  Choose some $Q \in \Gr{n}{n - 1}$, and denote by $\phi'$ the
restriction of $\phi$ to $\ell_1$-convex subsets of $Q$.  By inductive
hypothesis, there exist constants $c_0, \ldots, c_{n - 1}$ such that $\phi' =
\sum_{i = 0}^{n - 1} c_i \IV{i}$.  Thus, for all $\ell_1$-convex sets
$X \sub Q$,
\begin{equation}        \label{eq:hyper-comb}
\phi(X)
=
\sum_{i = 0}^{n - 1} c_i \IV{i}(X).
\end{equation}
Since $\phi$ is invariant under translations and coordinate permutations,
\eqref{eq:hyper-comb} holds for all $X \in \cvx{n}$ of dimension less than
$n$.  Put $\psi = \phi - \sum_{i = 0}^{n - 1} c_i \IV{i}$.  Then $\psi$ is a
continuous translation-invariant valuation, which we have just shown to be
simple.  Proposition~\ref{propn:volume} implies that $\psi = c_n \Vold{n}$ for
some $c_n \in \reals$.  Since $\Vold{n} = \IV{n}$, the result follows.  
\done
\end{proof}

\begin{cor}     \label{cor:hgs}
Let $\phi$ be a continuous invariant valuation on $\cvx{n}$, homogeneous of
degree $i \in \reals$.  Then $i \in \{0, \ldots, n\}$ and $\phi = c \IV{i}$
for some $c \in \reals$.  
\done
\end{cor}

Careful analysis of the proof of the theorem enables two refinements to be
made.  

First, when showing that every continuous invariant valuation $\phi$ on
$\cvx{n}$ was a linear combination of $\ell_1$-intrinsic volumes, we never
called on the fact that $\phi$ was invariant under reflections in coordinate
hyperplanes.  Second, we did not use the full strength of the assumption that
$\phi$ was continuous: only that $\phi$ was \demph{continuous from the
outside}, that is, $\displaystyle\lim_{Y \to X,\ Y \supseteq X} \phi(Y) =
\phi(X)$ for all $X \in \cvx{n}$.  (The essential point is that in
Proposition~\ref{propn:lc-outer}, the pixellated sets $\otr{X}{\lambda}$
contain $X$.)  But any linear combination of $\ell_1$-intrinsic volumes is
continuous and invariant under the full isometry group.  Hence:

\begin{cor}   \label{cor:improvement}
Let $\phi$ be a valuation on $\cvx{n}$, continuous from the outside and
invariant under translations and coordinate permutations.  Then $\phi$ is
continuous and invariant.  
\done
\end{cor}

\section{An analogue of Steiner's formula}
\label{sec:Ste}

The most obvious analogue of the classical Steiner formula would be an identity
of the form
\[
\Vol(X + \lambda \ball{n}) 
=
\sum_{i = 0}^n c_i \IV{i}(X) \lambda^{n - i}
\]
for $X \in \cvx{n}$ and $\lambda \geq 0$.  Here $\ball{n}$ denotes the closed
unit ball in $\ell_1^n$, and $c_0, \ldots, c_n$ are constants.  However, there
can be no such formula.  For if there were then $\Vol(\cblank + \ball{n})$
would be a valuation on $\cvx{n}$, and the following example demonstrates that
for general $n$, it is not.

\begin{example} \label{eg:ball-Steiner}
Let $X = [0, 1] \times \{0\}$ and $Y = \{0\} \times [0, 1]$, both subsets of
$\reals^2$.  Then $X$, $Y$, $X \cap Y$ and $X \cup Y$ are all $\ell_1$-convex
sets, and it is straightforward to calculate that
\[
\Vol((X \cap Y) + \ball{2})
<
\Vol((X + \ball{2}) \cap (Y + \ball{2})).
\]
But $(X \cup Y) + \ball{2} = (X + \ball{2}) \cup (Y + \ball{2})$ by
Proposition~\ref{propn:Min-cap}, so $\Vol(\cblank + \ball{2})$ is not a
valuation on $\cvx{2}$.
\end{example}

There is, however, a Steiner-type formula in which the role of the ball is
played by the cube $\ucube{n}$.

\begin{thm}
Let $X \in \cvx{n}$ and $\lambda \geq 0$.  Then for $0 \leq k \leq n$,
\[
\IV{k}(X + \lambda \ucube{n})
=
\sum_{i = 0}^k \binom{n - i}{n - k} \IV{i}(X) \lambda^{k - i}.
\]
In particular, 
\[
\Vol(X + \lambda \ucube{n})
=
\sum_{i = 0}^n \IV{i}(X) \lambda^{n - i}.
\]
\end{thm}

For the left-hand side of the first equation to be defined we need $X +
\lambda\ucube{n}$ to be $\ell_1$-convex.  This follows from
Proposition~\ref{propn:Min}.

\begin{proof}
We begin by showing that $\IV{k}(\cblank + \ucube{n})$ is a continuous
invariant valuation.  Proposition~\ref{propn:Min-cap} implies that it is a
valuation, since $\ucube{n}$ is an interval.  It is invariant, since
$\ucube{n}$ is invariant under isometries fixing the origin.  It is also
monotone, and therefore continuous by Corollary~\ref{cor:mono-cts}.

By the $\ell_1$ Hadwiger theorem~(\ref{thm:Had}), there are constants $c_i$
such that $\IV{k}(X + \ucube{n}) = \sum_{i = 0}^n c_i \IV{i}(X)$ for all $X
\in \cvx{n}$.  It follows that for $\lambda > 0$ and $X \in \cvx{n}$,
\[
\IV{k}(X + \lambda \ucube{n})
=
\lambda^k \IV{k} (\lambda^{-1} X + \ucube{n})
=
\sum_{i = 0}^n c_i \IV{i}(X) \lambda^{k - i}.
\]
The result follows on putting $X = \ucube{n}$, using
Example~\ref{egs:ivs}(\ref{eg:ivs-cube}). 
\done
\end{proof}

\section{An analogue of Crofton's formula}
\label{sec:Cro}

In this section and the next, we derive $\ell_1$ analogues of Euclidean
integral-geometric formulas.  The formal structure is similar to that in
Klain and Rota~\cite{KlRo}.

For $0 \leq k \leq n$, let $\Graff{n}{k}$ denote the set of $k$-dimensional
affine subspaces of $\reals^n$ parallel to some $k$-dimensional coordinate
subspace.  Each element of $\Graff{n}{k}$ is uniquely representable as $P +
q$ with $P \in \Gr{n}{k}$ and $q \in P^\perp$.

There is a natural measure on $\Graff{n}{k}$, invariant under isometries of
$\ell_1^n$.  Indeed, $\Graff{n}{k}$ is in canonical bijection with the
disjoint union $\coprod_{P \in \Gr{n}{k}} P^\perp$, each space $P^\perp$
carries Lebesgue measure $\Vol_{n - k}$, and summing gives the measure on
$\Graff{n}{k}$.

\begin{thm}     \label{thm:Cro}
Let $X \in \cvx{n}$.  Then for $0 \leq j \leq k \leq n$,
\begin{equation}        \label{eq:Crofton}
\int_{\Graff{n}{k}}
\IV{j}(X \cap A) \dee A
=
\binom{n + j - k}{j}
\IV{n + j - k}(X).
\end{equation}
In particular, for $0 \leq k \leq n$, the set 
\[
\bigl\{
A \in \Graff{n}{k} \such X \cap A \neq \emptyset
\bigr\}
\]
has measure $\IV{n - k}(X)$.
\end{thm}

\begin{proof}
We prove just the first statement, the second being the case $j = 0$.     

Write $\phi(X)$ for the left-hand side of~\eqref{eq:Crofton}.  Then $\phi$ is
a monotone invariant valuation since $\IV{j}$ is, and since the measure on
$\Graff{n}{k}$ is invariant.  It is therefore continuous, by
Corollary~\ref{cor:mono-cts}.  Moreover, $\phi$ is homogenous of degree $n + j
- k$, since for $\lambda > 0$ and $X \in \cvx{n}$,
\[
\phi(\lambda X)    
=
\int_{B \in \Graff{n}{k}} 
\IV{j}(\lambda X \cap \lambda B) 
\dee (\lambda B) 
=
\lambda^{j}\lambda^{n - k} \phi(X).
\]
So by Corollary~\ref{cor:hgs}, $\phi = c\IV{n + j -k}$ for some $c \in
\reals$.

By construction of the invariant measure, 
\[
\phi(\ucube{n})
=
\sum_{P \in \Gr{n}{k}}
\int_{P^\perp} 
\IV{j}(\ucube{n} \cap (P + q)) \dee q.
\]
For $P \in \Gr{n}{k}$ and $q \in P^\perp$, identifying $P^\perp$ with
$\reals^{n - k}$, we have
\[
\IV{j}(\ucube{n} \cap (P + q))
=
\begin{cases}
\IV{j}(\ucube{k}) = \binom{k}{j}&\text{if }q \in \ucube{n - k}      \\
0                               &\text{otherwise.}
\end{cases}
\]
From this it is straightforward to deduce the value of $c$.
\done
\end{proof}

A very similar argument, left to the reader, proves the following analogue of
Kubota's theorem~\cite{Kubo,KlRo}.  It can also be deduced directly from the
definition of the $\ell_1$-intrinsic volumes.  Corollary~\ref{cor:pjn-pres}
guarantees that the left-hand side is defined.

\begin{thm}
Let $X \in \cvx{n}$.  Then for $0 \leq j \leq k \leq n$,
\[
\sum_{P \in \Gr{n}{k}}
\IV{j}(\pi_P X)
=
\binom{n - j}{n - k} \IV{j}(X).
\]
\ \done
\end{thm}

\section{Analogues of the kinematic formulas}
\label{sec:kin}

The classical kinematic formulas concern the intrinsic volumes of sets $gX
\cap Y$, where $X$ and $Y$ are convex and $g$ is a Euclidean motion.
Remark~\ref{rmk:cogeodesic} suggests that fundamentally, one of $X$ and $Y$
should be regarded as geodesic and the other as \emph{co}geodesic, although in
the classical context the difference is invisible.  This leads us to expect
$\ell_1$ kinematic formulas in which $X$ is $\ell_1$-convex and $Y$ is an
interval.

To state the $\ell_1$ kinematic formulas, we first need a measure on the
isometry group $\hoaff{n}$ of $\ell_1^n$.  This is constructed as follows.
$\hoaff{n}$ has a subgroup $\ho{n}$, the $n$th hyperoctahedral group,
consisting of just the isometries fixing the origin.  Each element of
$\hoaff{n}$ is uniquely representable as $x \goesto h(x) + q$ with $h \in
\ho{n}$ and $q \in \reals^n$.  Being a finite group, $\ho{n}$ has a unique
invariant probability measure.  Taking the product of this measure with
Lebesgue measure on $\reals^n$ gives an invariant (Haar) measure on
$\hoaff{n}$.

We also need a result on products.  First observe that if $X \sub \reals^m$
and $Y \sub \reals^n$ are $\ell_1$-convex then $X \times Y$, viewed as a
subset of $\reals^{m + n}$, is also $\ell_1$-convex.

\begin{propn}   \label{propn:products}
Let $X \in \cvx{m}$, $Y \in \cvx{n}$, and $0 \leq k \leq m + n$.  Then 
\[
\IV{k}(X \times Y)
=
\sum_{i + j = k} \IV{i}(X) \IV{j}(Y)
\]
where $0 \leq i \leq m$ and $0 \leq j \leq n$ in the summation.
\end{propn}

\begin{proof}
By definition of the $\ell_1$-intrinsic volumes, 
\begin{align*}
\IV{k}(X \times Y)      &
=
\sum_{i + j = k}
\sum_{\ P \in \Gr{m}{i},\: Q \in \Gr{n}{j}} 
\Vol_{i + j}(\pi_{P \times Q}(X \times Y))      \\
&=
\sum_{i + j = k}
\sum_{\ P \in \Gr{m}{i},\: Q \in \Gr{n}{j}} 
\Vol_i(\pi_P X) \Vol_j(\pi_Q Y)
=
\sum_{i + j = k} \IV{i}(X) \IV{j}(Y).
\makebox[0em]{\hspace*{1.25em}\ensuremath{\Box}}
\end{align*}
\end{proof}

\begin{example} \label{eg:intvl-esp}
Let $I = I_1 \times \cdots \times I_n$ be a nonempty compact interval.  Then
$\IV{j}(I)$ is the $j$th elementary symmetric polynomial in the lengths of
$I_1, \ldots, I_n$ (also equal to the $j$th Euclidean intrinsic volume of
$I$).  
\end{example}

We now state the principal kinematic formula for $\ell_1^n$.

\begin{thm}     \label{thm:prin-kin}
Let $X \in \cvx{n}$ and let $I$ be a compact interval in $\reals^n$.  Then the
set
\begin{equation}        \label{eq:kine-set}
\bigl\{
g \in \hoaff{n}
\such
gX \cap I \neq \emptyset
\bigr\}
\end{equation}
has measure
\[
\sum_{i + j = n}
\binom{n}{i}^{-1} \IV{i}(X) \IV{j}(I).
\]
\end{thm}

\begin{proof}
Fix $I$.  Write $\phi(X)$ for the measure of the set~\eqref{eq:kine-set}:
then $\phi(X) = \int_{\hoaff{n}} \chi(gX \cap I) \dee g$.  This $\phi$ is a
monotone invariant valuation on $\cvx{n}$, and is therefore continuous by
Corollary~\ref{cor:mono-cts}.  So by the $\ell_1$ Hadwiger theorem,
\begin{equation}        \label{eq:kin-bival}
\phi
=
\sum_{i = 0}^n c_i \IV{i}
\end{equation}
for some real numbers $c_i$.  We compute $c_i$ by evaluating $\phi(\lambda
\ucube{n})$ for $\lambda > 0$.  By construction of the invariant measure on
$\hoaff{n}$,
\[
\phi(\lambda\ucube{n})
=
\frac{1}{\card{\ho{n}}}
\sum_{h \in \ho{n}}
\int_{\reals^n} 
\chi\bigl(
(h(\lambda\ucube{n}) + q) \cap I 
\bigr)
\dee q.
\]
But $\lambda\ucube{n}$ is $\ho{n}$-invariant, so $\phi(\lambda \ucube{n})$ is
the Lebesgue measure of the set
\begin{equation}        \label{eq:int-cuboid}
\{ q \in \reals^n 
\such 
(\lambda \ucube{n} + q) \cap I \neq \emptyset \}.
\end{equation}
If $I = \emptyset$ then the theorem holds trivially; suppose that $I \neq
\emptyset$.  Write $I = \prod_{r = 1}^n I_r$, and write $u_r$ for the length
of the interval $I_r$.  Then the set~\eqref{eq:int-cuboid} is a product of
intervals of lengths $\lambda + u_r$.  Hence
\begin{align*}
\phi(\lambda \ucube{n}) &
=
\prod_{r = 1}^n (\lambda + u_r) 
=
\lambda^n \prod_{r = 1}^n (1 + \lambda^{-1} u_r) 
=
\lambda^n \sum_{j = 0}^n \IV{j}(\lambda^{-1} I)
=
\sum_{j = 0}^n \IV{j}(I) \lambda^{n - j},
\end{align*}
using nonemptiness of $I$ and Example~\ref{eg:intvl-esp}.  On the other hand,
we may compute $\phi(\lambda \ucube{n})$ using~\eqref{eq:kin-bival}, and
comparing coefficients gives $c_i = \binom{n}{i}^{-1}\IV{n - i}(I)$.
\done
\end{proof}

Higher kinematic formulas for $\ell_1^n$ can be deduced from
Theorems~\ref{thm:Cro} and~\ref{thm:prin-kin} by an argument formally
identical to that in Section~10.3 of~\cite{KlRo}:

\begin{thm}
Let $0 \leq k \leq n$, let $X \in \cvx{n}$, and let $I$ be a compact interval
in $\reals^n$.  Then
\[
\int_{\hoaff{n}} \IV{k}(gX \cap I) \dee g
=
\sum_{i + j = n + k}
\binom{n}{i}^{-1}
\binom{j}{k} 
\IV{i}(X)
\IV{j}(I)
\]
where $0 \leq i \leq n$ and $0 \leq j \leq n$ in the summation.
\done
\end{thm}

% \ucontents{section}{References}


\begin{thebibliography}{10}

\bibitem{AlesOPM}
S.~Alesker.
\newblock On {P}.~{M}c{M}ullen's conjecture on translation invariant
  valuations.
\newblock {\em Advances in Mathematics}, 155:239--263, 2000.

\bibitem{AlesDTI}
S.~Alesker.
\newblock Description of translation invariant valuations on convex sets with
  solution of {P}.~{M}c{M}ullen's conjecture.
\newblock {\em Geometric and Functional Analysis}, 11:244--272, 2001.

\bibitem{AlesTVM}
S.~Alesker.
\newblock Theory of valuations on manifolds: a survey.
\newblock {\em Geometric and Functional Analysis}, 17:1321--1341, 2007.

\bibitem{AlesVMI}
S.~Alesker.
\newblock Valuations on manifolds and integral geometry.
\newblock {\em Geometric and Functional Analysis}, 20:1073--1143, 2010.

\bibitem{AlFe}
J.-C. {\'A}lvarez~Paiva and E.~Fernandes.
\newblock {C}rofton formulas in projective {F}insler spaces.
\newblock {\em Electronic Research Announcements of the American Mathematical
  Society}, 4:91--100, 1998.

\bibitem{Bern}
A.~Bernig.
\newblock Valuations with {C}rofton formula and {F}insler geometry.
\newblock {\em Advances in Mathematics}, 210:733--753, 2007.

\bibitem{BeFu}
A.~Bernig and J.~Fu.
\newblock {H}ermitian integral geometry.
\newblock {\em Annals of Mathematics}, 173:907--945, 2011.

\bibitem{FiWo}
E.~Fink and D.~Wood.
\newblock {\em Restricted-Orientation Convexity}.
\newblock Monographs in Theoretical Computer Science. Springer, Berlin, 2004.

\bibitem{FuKFI}
J.~Fu.
\newblock Kinematic formulas in integral geometry.
\newblock {\em Indiana University Mathematics Journal}, 39:1115--1154, 1990.

\bibitem{GroMSR}
M.~Gromov.
\newblock {\em Metric Structures for {R}iemannian and Non-{R}iemannian Spaces}.
\newblock Birkh{\"a}user, Boston, 2001.

\bibitem{Klai}
D.~A. Klain.
\newblock A short proof of {H}adwiger's characterization theorem.
\newblock {\em Mathematika}, 42:329--339, 1995.

\bibitem{KlRo}
D.~A. Klain and G.-C. Rota.
\newblock {\em Introduction to Geometric Probability}.
\newblock Lezioni Lincee. Cambridge University Press, Cambridge, 1997.

\bibitem{Kubo}
T.~Kubota.
\newblock {\"U}ber konvex-geschlossene {M}annigfaltigkeiten im
  $n$-dimensionalen {R}aume.
\newblock {\em Science Reports of the T{\^o}hoku University}, 14:85--99, 1925.

\bibitem{McMVET}
P.~Mc{M}ullen.
\newblock Valuations and {E}uler-type relations on certain classes of convex
  polytopes.
\newblock {\em Proceedings of the London Mathematical Society}, 35:113--135,
  1977.

\bibitem{McMCTI}
P.~McMullen.
\newblock Continuous translation-invariant valuations on the space of compact
  convex sets.
\newblock {\em Archiv der Mathematik}, 34:377--384, 1980.

\bibitem{Papa}
A.~Papadopoulos.
\newblock {\em Metric Spaces, Convexity, and Nonpositive Curvature}.
\newblock European Mathematical Society, Z{\"u}rich, 2005.

\bibitem{SchnCBB}
R.~Schneider.
\newblock {\em Convex Bodies: the {B}runn--{M}inkowski Theory}.
\newblock Encyclopedia of Mathematics and its Applications 44. Cambridge
  University Press, Cambridge, 1993.

\bibitem{SchnIVM}
R.~Schneider.
\newblock Intrinsic volumes in {M}inkowski spaces.
\newblock {\em Rendiconti del Circolo Matematico di Palermo, Serie II,
  Supplemento}, 50:355--373, 1997.

\bibitem{SchnIGP}
R.~Schneider.
\newblock On integral geometry in projective {F}insler spaces.
\newblock {\em Journal of Contemporary Mathematical Analysis}, 37:30--46, 2002.

\bibitem{SchnCMP}
R.~Schneider.
\newblock {C}rofton measures in projective {F}insler spaces.
\newblock In E.~L. Grinberg, S.~Li, G.~Zhang, and J.~Zhou, editors, {\em
  Integral Geometry and Convexity (Proc.\ Int.\ Conf., Wuhan, China, Oct.\
  2004)}, pages 67--98. World Scientific, New Jersey, 2006.

\bibitem{ScWi}
R.~Schneider and J.~A. Wieacker.
\newblock Integral geometry in {M}inkowski spaces.
\newblock {\em Advances in Mathematics}, 129:222--260, 1997.

\bibitem{Val}
F.~A. Valentine.
\newblock {\em Convex Sets}.
\newblock McGraw--Hill, New York, 1964.

\end{thebibliography}
\end{document}